\documentclass[12pt]{amsart}

\usepackage{comment,hyperref,color}

\usepackage{breakurl}   

\usepackage{enumitem} 

\usepackage{comment} 

\usepackage[dvips]{graphicx} 

\usepackage{setspace}  

\numberwithin{equation}{section}

\newtheorem{theorem}{Theorem}

\newtheorem{proposition}{Proposition}

\newtheorem{definition}{Definition}

\newcommand{\N}{\mathbb{N}}

\newcommand{\Q}{\mathbb{Q}}
\newcommand{\R}{\mathbb{R}}

\newcommand\astr{{{}^\ast\hspace*{-.5pt}\R}}
\newcommand\asta{{{}^\ast\hspace*{-3.5pt}A}}
\newcommand{\astb}{{}^{\ast}\hspace*{-2.5pt}B}

\newcommand{\power}{\mathcal{P}}

\newcommand{\ns}[1]{{}^\ast\hspace{-1pt}#1}

\newcommand{\civita}{\mathcal{R}}
\newcommand{\lang}{\mathcal{L}}

\newcommand{\m}{m}
\newcommand{\leb}{\mu}

\newcommand{\mode}{\hspace{-8pt}\mod}

\newcommand{\reg}{\textbf{Reg}}
\newcommand{\sy}{\textbf{Sy\hskip-.7pt}}
\newcommand{\un}{\textbf{U\hskip-1.2pt n\hskip.2pt}}
\newcommand{\co}{\textbf{Co}}

\newcommand{\tot}{\textbf{Tot}}
\newcommand{\olim}{\lim_{\lambda \uparrow \Omega}}

\newcommand\astf{{{}^{\ast}\hspace*{-3pt}f}}

\title [Internality, transfer, and infinitesimal modeling]
{Internality, transfer, and infinitesimal modeling of infinite
processes}

\author{Emanuele Bottazzi}
\address{E.~Bottazzi\\
	Department of Civil Engineering and Architecture\\
	University of Pavia\\
	Via Adolfo Ferrata 3, 27100 Pavia, Italy, orcid 0000-0001-9680-9549}
\email{emanuele.bottazzi@unipv.it, emanuele.bottazzi.phd@gmail.com}

\author{Mikhail G. Katz}
\address{M.~Katz\\
	Department of Mathematics\\
	Bar Ilan University\\
	Ramat Gan 5290002 Israel, orcid 0000-0002-3489-0158}
\email{katzmik@macs.biu.ac.il}

\date{\today}

\begin{document}


\begin{abstract}
A probability model is underdetermined when there is no rational
reason to assign a particular infinitesimal value as the probability
of single events.  Pruss claims that hyperreal probabilities are
underdetermined.  The claim is based upon external hyperreal-valued
measures.  We show
that internal hyperfinite measures are not underdetermined.  The
importance of internality stems from the fact that Robinson's transfer
principle only
applies to internal entities.  We also evaluate the claim that
transferless ordered fields (surreals, Levi-Civita field, Laurent
series) may have 
advantages over hyperreals in probabilistic modeling.  We show that
probabilities developed over such fields are less expressive than
hyperreal probabilities.
\end{abstract}



\maketitle
\tableofcontents

\section{Introduction}

Alexander Pruss (AP) claims in \cite{Pr18} that hyperreal
probabilities are underdetermined, meaning that, given a model, there
is no rational reason to assign a particular infinitesimal value as
the probability of a single event.  AP's underdetermination claim
hinges upon the following:
\begin{enumerate}
\item
examples of uniform processes that allegedly do not allow for a
uniquely defined infinitesimal probability for singletons, and
\item
a pair of theorems asserting that for every hyperreal-valued
probability measure there exist uncountably many others that induce
the same decision-theoretic preferences.
\end{enumerate}
In \cite{BK20a} we presented our main arguments highlighting some
hidden biases in the representation of some infinite processes such as
lotteries, coin tosses, and other infinite processes.

In Section~\ref{s24} we analyze some Archimedean and non-Archimedean
models of Prussian spinners (rotating pointers).  In particular, we
will show that hyperfinite models of the spinners are not
underdetermined.  Moreover, we will discuss additional constraints
that may narrow down the choice of infinitesimal probabilities, as
proposed by AP himself in \cite{Pr18}, Section\;3.3 as well as 3.5
(``Other putative constraints'').

We show that the additional probabilities introduced in AP's theorems
are all external; see Section \ref{sec underdetermination}.  Since
external functions do not obey the transfer principle of Robinson's
framework, these additional measures are inferior to internal ones
when it comes to modeling using hyperreal fields.  Thus AP's external
measures are \emph{parasitic} in the sense of Clendinnen (\cite{Cl89},
1989),%
\footnote{\label{f5}See \cite{BK20a}, note 2.}
and fail to establish underdetermination.  In the light of our
analysis, AP's theorems can be framed as a warning against the use of
external probabilities that do not satisfy the transfer principle.
Thus, among the additional constraints that, in AP's words, ``may help
narrow down the choice of infinitesimal probabilities''
(\cite[Section\;3.3]{Pr18}), the first choice should be the
\emph{internal} constraint; see Sections~\ref{s45} and \ref{notunder}
as well as \cite{BK20a}, Section\;3.2.  AP speculates that
\begin{quote}
[An] approach in terms of imprecise probabilities, where a
\emph{family} of hyperreal-valued probability functions is assigned as
in Sect.\;6.1.2 of Benci et al.\;(2018), may well be able to escape
the underdetermination worries.  (\cite[Section\;1]{Pr18}; emphasis in
the original)
\end{quote}
AP's wording implies that a \emph{single} hyperreal-valued probability
is unable to escape such `worries.'  However, we show that a single,
internal hyperreal-valued probability does escape Prussian worries.

AP also speculated that other non-Archimedean fields may bestow
advantages as compared to Robinson's framework.  In Section \ref{s6}
we argue that, contrary to AP's speculation, infinitesimal
probabilities in other familiar non-Archimedean extensions of the real
numbers are less expressive than real-valued probabilities.  This is
due to the absence of a uniform way of extending most real functions
to these fields so as to preserve their elementary properties, and to
the limitations of the measure theory currently developed for these
fields.  From the viewpoint of Robinson's framework, such limitations
are due to the unavailability of a transfer principle for these
fields.

\section{Archimedean and otherwise models of Pruss spinners} 
\label{sec spinners}
\label{s24}

We analyze AP's non-Archimedean models of the spinners, and show how
hyperfinite non-Archimedean models avoid the claimed
underdetermination problem.  Starting from some of these hyperfinite
models, it is also possible to define an infinitesimal probability
with a standard sample space and taking values in a field of
hyperreals~$\astr$, thus addressing AP's original model (see Section
\ref{s45}).  In this section we examine one particular example of the
failure of AP's underdetermination charge for hyperreal probabilities.
Our rebuttal of his claim that \emph{every} non-Archimedean
probability is underdetermined appears in Section~\ref{sec
underdetermination}.

AP provides the following example (see \cite[Section\;3]{Pr18}) of a
pair of uniform processes with a common sample space:
\begin{enumerate}
\item the first process is a spinner that is designed to stop
uniformly at an angle~$\theta \in [0,360)$;
\item the second process is a spinner that is designed to stop
uniformly at an angle~$\theta \in [0,360)$, but once it stops, the
angle obtained is doubled.
\end{enumerate}

AP then shows that the outcome of both spinners can be represented
by a uniform probability over the real interval~$[0,360)$.  Having
introduced this Archimedean model, AP alleges that the description
of the spinners by means of a hyperreal-valued probability depends
upon the underlying mechanism governing each spinner.  He concludes
that
\begin{quote}
we cannot simply specify an infinitesimal [probability] by saying that
it is whatever is the probability of a uniform process hitting a
particular point. For what that probability is -- assuming it is
infinitesimal and not zero -- depends on details of the process that
go beyond uniformity. \cite[Section\;3.2]{Pr18}
\end{quote}
To refute AP's claim, we will show that his conclusion is only due to
inappropriate choices he makes in the non-Archimedean description of
the uniform process, and that indeed appropriate non-Archimedean
models do not depend upon irrelevant details.

We observe that AP appears to extrapolate to every
non-Archi\-me\-dean model some properties of the continuous model
obtained from the Lebesgue measure.  For instance, he assumes that in
every infinitesimal description of the spinners, the sample space~$S$
must have the property that if~$x \in S$, then also~$\frac{x}{2}\in
S$.  This property is satisfied if~$S$ is an interval of~$\Q$ or
of~$\R$, but the spinners can also be described by means of a
hyperfinite set~$S_H$ that is discrete in an appropriate
sense.%
\footnote{We have already argued that a given physical process can be
given distinct mathematical representations; see \cite{BK20a},
Section\;2.3.  Thus we reject the claim that changing the sample
space, as we have proposed here, changes also the underlying process.}
The representations of the uniform processes by means of a hyperfinite
sample space~$S_H$ enable us to specify a unique uniform probability
measure over~$S_H$ that is also regular (see also Section \ref{sec
hyperfinite spinners}). Thus for such hyperreal probabilities no
underdetermination occurs.  In Sections\;\ref{s42} and \ref{s53} we
argue that this result applies also to non-Archimedean probabilities
that satisfy the~$\Omega$-limit axioms of Benci et al., since these
probabilities can be obtained as the restriction of suitable
hyperfinite measures.


Consequently, in his intuitive arguments for his underdetermination
claim AP relies on additional hypotheses on the non-Archimedean
models 
that are not satisfied by many relevant hyperreal probabilities.

In Archimedean mathematics, the two spinners discussed by AP can be
described by a family of discrete models and by a continuous model, as
we explain in Sections~\ref{sec archimedean spinners} and \ref{sec
archimedean spinners 2}.  In Sections\;\ref{sec hyperfinite spinners}
and \ref{s44} we show how internal hyperfinite models of AP's spinners
avoid the alleged drawbacks.  Finally, in Section\;\ref{s45}, we show
how properly defined hyperfinite measures can approximate to varying
degrees other properties of the models based upon the Cantor--Dedekind
representation of the continuum.

\subsection{Discrete models of the spinners}
\label{sec archimedean spinners}

The discrete models for the first spinner can be obtained as
follows. Let~$n\in\N$ and define
\[
S_n=\big\{0,\ldots,\tfrac{k\pi}{n},\ldots,\tfrac{(2n-1)\pi}{n}\big\}
\subseteq [0,2\pi).
\]
One can imagine the set~$S_n$ as representing the points on the circle
obtained from rotations by integer multiples of the
angle~$\frac{\pi}{n}$.  A uniform probability over~$S_n$ is given
by~$P_n(A)=\frac{|A|}{|S_n|}=\frac{|A|}{2n}$. The probability measure
$P_n$ has the following four properties:
\begin{itemize}
	\item[\tot]~$P_n$ is total, in the sense that it is defined on
	the powerset~$\power(S_n)$ of the sample space~$S_n$;
	\item[\un$_n$]~$P_n$ is uniform in the sense that for every~$A, B \subseteq S_n$ with~$|A|=|B|$, one has also~$P_n(A)=P_n(B)$;
	\item[\reg]~$P_n$ is regular;%
\footnote{\label{footnote regular}By \emph{regular} we mean that the
measure~$P_n$ assigns probability~$0$ only to the empty event. This
should not be confused with the notion of \emph{regular measure},
i.e. of a measure~$\nu$ such that~$\nu(A) = \sup\{ \nu(F) \colon F
\subseteq A$ is compact and measurable$\} = \inf\{\nu(G)\colon G
\supseteq A$ is open and measurable$\}$.}
	\item[$\neg$\sy$_{\R}$]~$P_n$ is not symmetric with respect to
	rotations by an arbitrary real angle, since if~$\theta\in\R$
	in general~$\frac{k\pi}{n}+\theta \mod 2\pi \not \in S_n$.
\end{itemize}
In addition, as a consequence of property \un$_n$, the measure~$P_n$
satisfies the following discrete symmetry condition:
\smallskip
\begin{itemize}
	\item [\sy$_n$]~$P_n$ symmetric with respect to rotations by
	multiples of~$\frac{\pi}{n}$, i.e.\ if~$0 \leq \theta \leq
	2n-1$ is an integer, then
\[
P\big(\big\{\tfrac{k_1\pi}{n},\ldots,\tfrac{k_j\pi}{n}\big\}\big)
=P\left(\left\{ \tfrac{k_1\pi}{n}+\theta\tfrac{\pi}{n} \mode
2\pi,\ldots,\tfrac{k_j\pi}{n}+\theta\tfrac{\pi}{n}\mode2\pi\right\}\right).
\]
\end{itemize}
\smallskip

The discrete models~$Q_n$ for the second spinner can be obtained from
the measure~$P_{2n}$ as follows:
\[
Q_n(\{x\})
=P_{2n}(\{\tfrac{x}{2}\})+P_{2n}(\{\tfrac{x}{2}+\pi\})=2\tfrac{1}{4n}=\tfrac{1}{2n}=P_n(\{x\})
\]
for all~$x \in S_n$.  From the previous equation it is readily seen
that~$Q_n$ agrees with~$P_n$ over~$S_n$, even though it is obtained
from another description of the uniform process.  As a
consequence,~$Q_n$ has the properties \un$_n$, \reg,~$\neg$\sy$_{\R}$
and \sy$_n$ already discussed for~$P_n$.

As already observed by Bascelli et al.\;(\cite{14a}, 2014), with a
choice of some sufficiently large~$n$, the models~$P_n$ and~$Q_n$
might already be sufficient for many practical purposes, since
\begin{quote}
all physical quantities can be entirely parametrized by the usual
rational numbers alone, due to the intrinsic limits of our capability
to measure physical quantities. \cite[p.\;853]{14a}
\end{quote}
This position is shared by Herzberg \cite[Section 1]{He07}.
Nevertheless, we agree with the common mathematical practice that it
is more convenient to introduce some level of idealisation for the
description of~$P_n$ and of~$Q_n$, especially as~$n$ becomes very
large.

\subsection{Spinner models based on the Lebesgue measure}
\label{sec archimedean spinners 2}

A typical idealisation is the use of continuous models based upon the
real numbers.  A continuous model for the first spinner possessing the
following property, \un, of uniformity:
\begin{enumerate}
	\item[\un] if two intervals have the same length, then they
	have the same probability
\end{enumerate}
\smallskip
is obtained by using the Lebesgue measure~$\mu$ over the real interval
$[0, 2\pi)$.  In this model, events are the measurable subsets of
$[0,2\pi)$ and the probability of an event~$A$ is defined as
$P_{\leb}(A)=\frac{\leb(A)}{\leb([0,2\pi))}=\frac{\leb(A)}{2\pi}$.

Notice also that the choice of~$[0, 2\pi)$ as the sample space is
arbitrary.  One could equally well measure angles with 
%
%
arc degrees instead of radians.  In this case, the sample space would
be the real interval~$[0,360)$ and the probability of an event~$A$
would be defined as~$P_{360}(A) =
\frac{\leb(A)}{\leb([0,360))}=\frac{\leb(A)}{360}$.  Despite the
arbitrariness of the choice of the sample space, one obtains the
compatibility conditions
\begin{equation}\label{compatibility 1}
	P_{360}(A) = P_{\leb}\left(\left\{x \in [0,2\pi) \colon
	\tfrac{360}{2\pi} x \in A \right\}\right)
\end{equation}
for all measurable sets~$A \subseteq [0,360)$, and 
\begin{equation}\label{compatibility 2}
P_{\leb}(B) = P_{360}\left(\left\{x \in [0,360) \colon
\tfrac{2\pi}{360} x \in B \right\}\right)
\end{equation}
for all measurable sets~$B \subseteq [0,2\pi)$.  From now on, we will
refer mostly to the probability measure~$P_{\leb}$, but the discussion
remains valid also for~$P_{360}$.  The probability measure~$P_{\leb}$
has different properties from its discrete counterparts~$P_n$.
Namely, it has the following three properties:
\begin{itemize}
	\item[\un]~$P_{\leb}$ is uniform;
	\item[$\neg$\reg]~$P_{\leb}$ is not regular;
	\item[\sy$_\R$]~$P_{\leb}$ is symmetric with respect to
	rotations by an arbitrary real angle.%
\footnote{Observe that we have not taken a position on whether the
continuous models satisfy {\tot} or\,~$\neg$\tot: if one rejects the
full Axiom of Choice, then the Lebesgue measure can be total; see
Solovay (\cite{So70}, 1970).  However, by assuming a sufficiently
strong choice principle it is possible to prove that there are sets
that are not Lebesgue measurable, so that~$P_{\leb}$ would not be
total. Recall also that without some weak form of the Axiom of Choice
it is not possible to prove that the Lebesgue measure is countably
additive. For more details, we refer to Fremlin (\cite{fremlin}, 2008)
and to Bottazzi et al.\;(\cite{19c}, 2019, Section 5).}
\end{itemize}
Notice also that for all~$n \in \N$, and for all~$h < k \leq 2n$, we
have
\[
P_{\leb}\left( \left[ \tfrac{h\pi}{n}, \tfrac{k\pi}{n}\right)\right)=
P_{n}\left(\left\{\tfrac{h\pi}{n},\ldots,\tfrac{(k-1)\pi}{n}\right\}\right).
\]
The previous equality can be interpreted as a coherence
between~$P_{\leb}$ and~$P_{n}$ over arcs with endpoints in~$S_n$.

Similarly, the second spinner can be modeled by using the Lebesgue
measure.  The probability of an event~$A$ is
\begin{align*}
	Q_{\leb}(A)&=\frac{\leb(\{x\colon 2x \mode 2\pi\in
	A\})}{\leb([0,2\pi))}\\ &=\frac{\leb(\{x/2\colon x\in
	A\})}{2\pi} + \frac{\leb(\{x/2+\pi\colon x\in A\})}{2\pi}.
\end{align*}
Since the Lebesgue measure is translation invariant, we have the
equality
$$
\frac{\leb(\{x/2\colon x\in A\})}{2\pi} = \frac{\leb(\{x/2+\pi\colon x\in A\})}{2\pi},
$$
from which we deduce
$$
Q_{\leb}(A) = 2\frac{\leb(\{x/2\colon x\in A\})}{2\pi} = 2P_{\mu}(\{x/2\colon x\in A\})=P_{\leb}(A).
$$

The discrete models and the continuous model capture different aspects
of the spinners. The choice of which model to use in specific
circumstances ought to reflect the relevant properties of each
situation.  For instance, one should take into account that, as is well
known, the properties \un, \reg\ and \sy$_\R$ cannot be
simultaneously satisfied by any real-valued measure defined over an infinite set~$\Omega$.
This imposes
some limitations on the scope of~$P$.  For instance, it is not
possible to define from~$P$ a conditional probability with respect to
every measurable event~$B$.  This occurs since 
%
%
the common definition of conditional probability 
%
%
$P(A|B) = \frac{P(A \cap B)}{P(B)}$ requires~$B$ to have positive
measure.

\subsection
{Hyperfinite models of the spinners}
\label{sec hyperfinite spinners}

The continuous model based upon the Lebesgue measure is not the only
possible idealisation of the two discrete models for the spinners.  By
using hyperfinite representations, it is possible to define models
that satisfy \tot\ and \reg, and that approximate to varying degrees
the properties \un\ and \sy$_\R$.

A simple model of the first spinner is obtained by taking as sample
space the set~$S_n$ for some infinitely large hypernatural~$n$.%
\footnote{Such models exist thanks to the transfer principle of
Robinson's framework. Moreover, transfer ensures that the properties
of the finite models discussed in Section\;\ref{sec archimedean
spinners} are shared by the hyperfinite models presented here.}
The corresponding probability measure~$P_n$ has the properties \tot,
\un$_n$, \reg,~$\neg$\sy$_\R$, \sy$_n$ described for the discrete
models of the spinner.%
\footnote{Property \tot\ should be interpreted as follows:~$P_n$ is
defined on every \emph{internal} set~$A \subseteq S_n$.  Similarly,
property \un$_n$ should be interpreted as follows: for every
\emph{internal}~$A, B \subseteq S_n$ with~$|A| = |B|$, one has also
$P_n(A) = P_n(B)$. The symbol~$| A |$ denotes the element of~$\ns{\N}$
corresponding to the \emph{internal cardinality} of the internal set
$A$.}
If~$n=m!$ for some infinite hypernatural~$m$, then \sy$_n$ implies
\sy$_\Q$, that is symmetry with respect to rotations by any rational
angle. As observed for instance by Benci et al.\ \cite[p.\;5]{bbd},
$P_n$ cannot satisfy the properties \un\ and \sy$_\R$; however it has
the coherence property
\begin{quote}
\textbf{\co}\;~$P_n(\asta\cap S_n)\approx P_{\mu}(A)$ for each
measurable~\hbox{$A\subseteq[0,2\pi)$},%
\footnote{Here~$\approx$ denotes the relation of infinite proximity,
i.e., the relation of being infinitely close.}
\end{quote}
that implies the weak form of uniformity
\[
P_n([a,b)) \approx P_n([c,d)) \text{ whenever } b-a = d-c,
\]
and the weak form of symmetry
\[
P_n(\asta\cap S_n) \approx P_n(\ns\hskip-1pt{\{x+\theta \mode 2\pi\colon x \in
A\}}\cap S_n)
\]
for every measurable~$A \subseteq [0,2\pi)$ and for every real~$\theta
\in [0,2\pi)$.

As already discussed in the Archimedean discrete models, the second
spinner can be represented by the hyperfinite probability~$Q_n$ for an
infinite hypernatural~$n$.  This probability measure shares the
properties \tot, \un$_n$,~$\neg$\sy$_\R$, \sy$_n$ and \co\ with the
first spinner; moreover one has~$Q_n(\{x\})=P_n(\{x\})$, as in the
Archimedean models.

Recall that the continuous description of the spinners based upon the
real Lebesgue measure could be formulated with different choices of
the sample space (e.g., the interval~$[0,2\pi)$ or the interval
$[0,360)$) but it is independent of such choices in the sense
expressed by equations\;\eqref{compatibility 1} and
\eqref{compatibility 2}.  For the hyperfinite measures~$P_n$ we have a
similar property, expressed by the following compatibility condition:
whenever~$\gcd(m,n) = \min\{m,n\} = n$, we have
\begin{equation}
\label{ns compatibility 1}
P_{n}(A) = P_{m}\left(\left\{x \in S_m\colon \tfrac{m}{n} x \mode 2\pi
\in A \right\}\right)
\end{equation}
for all~$A \subseteq S_n$. 
%
%

Equality \eqref{ns compatibility 1} 
%
%
is the discrete counterpart of equations \eqref{compatibility 1} and
\eqref{compatibility 2}.  The main differences between the continuous
equations and the hyperfinite one is that in the real case there is a
bijection between the sets~$[0,2\pi)$ and~$[0,360)$ given by the map
$x \mapsto \frac{360}{2\pi}x$, while in the discrete case (either
finite or hyperfinite) there is no bijection between the sets~$S_n$
and~$S_m$ unless~$n=m$.%
\footnote{We remark that there is no \emph{internal} bijection between
$S_n$ and~$S_m$. If one drops the requirement that the bijection must
be internal, then it is possible to find external bijections
witnessing that the external cardinality of~$S_n$ is equal to the
external cardinality of~$S_m$ whenever both~$n$ and~$m$ are
infinite. However, the existence of external bijections between~$S_n$
and~$S_m$ has no bearing our argument, that is instead based upon the
internal cardinality.  For a discussion on the limited relevance of
external objects and functions in hyperreal models see
Section~\ref{sec underdetermination} and note\;\ref{foot37}, as well
as \cite{BK20a}, Sections 3.2 and 3.4.}
On the other hand, if~$\gcd(m,n) = \min\{m,n\} = n$, then there is a
$\frac{m}{n}$-to-one correspondence between~$S_m$ and~$S_n$; together
with properties \un$_n$ and \un$_m$ of the measures~$P_n$ and~$P_m$,
this is sufficient to entail the relation~\eqref{ns compatibility 1}.

\subsection
{Hidden assumptions on infinitesimal models of spinners}
\label{s44}

In his Section\;3.2, AP seems to suggest that in Robinson's framework
every pair of probability measures~$P$ and~$Q$ for the first and the
second spinner, respectively, satisfy the following conditions:
\begin{enumerate}
[label={(Pr\,\theenumi)}]
\item 
\label{l1}
$P$ and~$Q$ have the same sample space;
\item
\label{l2}
$P$ and~$Q$ should assign the same probability to singletons;
\item
\label{l3}
$Q(\{x\})=2P(\{x/2\})$.
\end{enumerate}
From properties \ref{l2} and \ref{l3} AP obtains his \emph{underdetermination} claim \cite[Section\;3.2]{Pr18}.

However, as we have seen in the discussion of the discrete models, if
one wishes to retain property \ref{l1} in a hyperfinite setting,
%
%
the requirement \ref{l3} must be replaced by a subtler condition.  In
particular, for the hyperfinite probabilities~$P_n$ and~$Q_n$ defined
on~$S_n$ we have seen that~$Q_n$ should be defined in terms of
$P_{2n}$ as follows:
\begin{equation}
\tag{$3'$}
\label{3prime}
Q_n(\{x\}) = P_{2n}(\{\tfrac{x}{2}\}) +
	P_{2n}(\{\tfrac{x}{2}+\pi\ \ \mode 2\pi\}) =
	2P_{2n}(\{\tfrac{x}{2}\}),
\end{equation}
since for some~$x \in S_n$ the number~$x/2$ does not belong to the
sample space~$S_n$ but does belong to~$S_{2n}$.  We have already
argued that definition \eqref{3prime} implies that~$Q_n(\{x\})$ is
indeed equal to~$P_n(\{x\})$ for every~$x \in S_n$.  This example
shows how internal hyperfinite models 
%
%
of the two spinners are sufficient to refute AP's claim that the
infinitesimal probability of a pair of uniform processes on the same
outcome space allegedly ``depends on details of the process that go
beyond uniformity'' \cite[Section\;3.2]{Pr18}.

A similar rebuttal applies to AP's description of the pair of uniform
lotteries over~$\N$ in \cite[Section\;4.1]{Pr18}.

\subsection
{Refining the models by means of additional constraints}
\label{s45}

AP notes that
\begin{quote}
[T]here are potential types of constraints on the choice of
infinitesimals that may help narrow down the choice of infinitesimal
probabilities.  One obvious constraint is formal: the axioms of
finitely additive probability.  A second type of constraint is
\emph{match} between the extended real probabilities for a problem and
the corresponding classical real-valued probabilities.  (\cite[Section
3.3]{Pr18}; emphasis added)
\end{quote}
Thus, AP envisions the possibility of introducing additional
constraints that may narrow down the choice of infinitesimal
probabilities.  He even evokes more specifically the possibility of
exploring a ``match'' between extended-real probabilities and
classical real-valued probabilities.  To practitioners of mathematics
in Robinson's framework it would be natural to interpret such a
``match'' in terms of the constraint of being \emph{internal}.%
\footnote{Internal sets are elements of ($\ast$-extensions of)
classical sets; see \cite{BK20a}, Section\;3.1.  Robinson's book
(\cite{Ro66}, 1966) deals with \emph{internal} entities
systematically.}
%
%
By envisioning the possibility of introducing further constraints, AP
opens the door to introducing the internal condition.  The internal
condition would in any case be the obvious first choice for a
practitioner of mathematics in Robinson's framework.  It is a
constraint that AP failed to consider.  Similar shortcomings of Elga's
analysis in (\cite{El04}, 2004) were signaled by Herzberg
(\cite{He07}, 2007).

In addition, it is possible to introduce further constraints.  As an
example, we propose two different hyperfinite representations of the
uniform spinners.  Hyperfinite entities, being internal, satisfy the
transfer principle. Moreover, each of the following hyperfinite models
will have various properties that improve upon \tot, \un$_n$,
\reg,~$\neg$\sy$_\R$ and \co\ of the measure~$P_n$.

The main result of Benci et al.\;\cite{bbd} entails that there exists
\begin{itemize}
	\item an algebra of Lebesgue measurable sets~$\mathfrak{B}
	\subseteq \power([0,2\pi))$ such that for every~$A \in
	\mathfrak{B}$, either~$A = \emptyset$ or~$P_{\leb}(A) \ne
	0$;%
\footnote{An interesting and nontrivial choice is the ring of finite
unions of intervals of the form~$[a,b)$. For more details, see Benci
et al.\;(\cite{bbd2}, 2015, p.\;43).}
	\item a hyperfinite set~$\Omega\subseteq{}^\ast[0,2\pi)$ such that for every~$x \in [0,2\pi)$,~$\ns{x} \in \Omega$;
	\item a hyperfinite probability measure~$P_{\Omega}$ over~$\Omega$
\end{itemize}
with the properties \tot, \un$_ {|\Omega|}$, \reg,~$\neg$\sy$_{\R}$, \co\ and the
following additional property:
\begin{itemize}
	\item[\un$_\mathfrak{B}$]~$P_{\Omega}$ is uniform over sets of
	$\mathfrak{B}$, i.e.,
	$P_{\Omega}(\asta\cap\Omega)=
	P_{\Omega}(\astb\cap\Omega)$ whenever~$A, B \in \mathfrak{B}$
	and~$P_{\leb}(A)=P_{\leb}(B)$.
%
%
\end{itemize}
As a consequence of \un$_{|\Omega|}$,~$P_{\Omega}$ satisfies \un$_{n}$
for all~$n \in \N$.  Moreover, the property~\un$_\mathfrak{B}$ implies
that~$P_{\Omega}$ is symmetric with respect to rotations by an
arbitrary real angle, but only on the nonstandard extensions of sets
in~$\mathfrak{B}$.  For more details, we refer to \cite{bbd} and
\cite[Section\;3]{bbd2}, where the example of the Lebesgue measure is
discussed in detail.

The hyperfinite measure~$P_{\Omega}$ can also be used to define a
non-Archi\-me\-dean probability over the sample space~$[0,2\pi)$
%
%
by setting
\begin{equation}
\label{e44}
P(A)=P_\Omega(\asta\cap\Omega) \quad \text{for each} \quad
A\subseteq[0,2\pi).
\end{equation}
The hypothesis that for every~$x \in [0,2\pi)$,~$\ns{x} \in \Omega$
ensures that $P_\Omega(x)=\frac{1}{|\Omega|}$ for
every~$x\in[0,2\pi)$.  Thus the measure $P$ of \eqref{e44} is regular
and uniform in the sense that it assigns the same nonzero probability
to singletons.

In Section \ref{s42} we argue that non-Archimedean probabilities that
satisfy the Omega-limit axiom of Benci et al.\ can be obtained in a
similar way, i.e., by restricting a suitable hyperfinite probability.
In this external measure over~$[0,2\pi)$, the choice for the
probability of a singleton is uniquely determined from the underlying
hyperfinite set~$\Omega$, and it is not arbitrary as argued by AP.

By using an earlier result by Wattenberg \cite{watt}, it is possible
to obtain another hyperfinite set~$\Omega_H$ and a hyperfinite
probability measure~$P_H$ over~$\Omega_H$ that has an additional
advantage.  Namely, it is coherent not only with the Lebesgue measure,
but also with the Hausdorff~$t$-measures.

Recall that the Hausdorff outer measure of order~$t$ of a set~$A
\subseteq \R$ is defined as
\[
\overline{H}^t(A)=\lim_{\delta \rightarrow 0}\left\{\sum_{n\in\N}
\lambda(I_n)^t \colon A \subseteq \bigcup_{n\in\N} I_n,
\lambda(I_n)<\delta\right\}.
\]
The Hausdorff measure~$H^t$ is the restriction of the outer measure
$\overline{H}^t$ to the~$\sigma$-algebra of measurable subsets of
$\Omega$.

The probability measure~$P_H$ over~$\Omega_H$ has the properties \tot,
\mbox{\un$_{|\Omega_H|}^{\phantom{I}}$}, \reg,~$\neg$\sy$_{\R}$ as
well as the following property:
\begin{itemize}
	\item[\co$_H$] if~$B$ is a Borel set with finite nonzero
	Hausdorff~$t$-measure for some real~$t \in [0,+\infty)$, then
	for every~$H^t$-measurable set~$A \subseteq \Omega_H$ the sets
	$A_H =\asta \cap \Omega_H$ and~$B_H = \astb \cap \Omega_H$
	satisfy the relations
	$$\frac{H^t(A \cap B)}{H^t(B)} \approx \frac{P_H(A_H\cap
	B_H)}{P_H(B_H)} = P_H(A_H|B_H).$$
\end{itemize}
Notice that, by taking~$t=1$ and~$B=[0,2\pi)$, \co$_H$ implies \co.

Each of the probability measures~$P_n$,~$P_\Omega$ and~$P_H$ described
above models different aspects of the spinner.  In particular, they
are all uniform, regular and total probability measures; moreover they
approximate to varying degrees the properties of the real continuous
model based upon the Lebesgue measure or upon the Hausdorff measure.

Furthermore, the probability~$P_\Omega$ shows how the critique that
non-Archimedean probabilities do not preserve intuitive symmetries,
presented by AP in \cite[Section\;3.3]{Pr18}, can be addressed by
means of a suitable hyperfinite model.
%
%

Meanwhile, the probability~$P_H$ shows that hyperreal-valued
probability measures can be used simultaneously to represent the
uncountably many Hausdorff~$t$-measures.  The strength of these kinds
of hyperfinite models is not discussed by AP, who only considers the
coherence of a non-Archimedean measure with a single real-valued
measure.

\section{Are infinitesimal probabilities underdetermined?} 
\label{sec underdetermination}

In Section \ref{sec spinners} we presented an analysis of AP's
non-Archimedean modeling.  Now we turn to his
\emph{underdetermination} theorems.  AP's first theorem is based upon
the following construction.  Let~$P$ be a probability function with
values in a non-Archimedean ordered extension of~$\R$.  AP sets
\begin{equation}
\label{e41}
P_\alpha(A) = \mathrm{St}(P(A)) + \alpha(P(A)-\mathrm{St}(P(A)))
\end{equation}
for every~$\alpha > 0$.%
\footnote{AP assumes in addition that~$\alpha \in \R$.  However his
results are still valid for every positive finite~$\alpha \in \astr$.}
AP's first theorem expresses the fact that~$P_\alpha$ is infinitely
close to~$P$ and satisfies the same inequalities as~$P$ does.  This
result can be interpreted as the fact that the probabilities~$P$ and
$P_\alpha$ induce the same comparisons between events.

AP's second theorem asserts a similar result for probability measures
that satisfy the~$\Omega$-limit axiom of Benci et al.
%
%
In this case, starting from a probability that satisfies
the~$\Omega$-limit axiom (see Section~\ref{s42}) and from an
automorphism~$\phi$ of $\astr$ that fixes only the standard real
numbers, it is possible to define a new probability
\begin{equation}
\label{e411}
P_\phi = \phi \circ P
\end{equation}
 that yields the same comparisons
between events 
as~$P$, and that still satisfies the
$\Omega$-limit axiom.

Both results are technically correct, but what AP fails to mention is
that, if~$P$ is internal, then the additional probabilities~$P_\alpha$
and~$P_\phi$ are all external whenever they differ from~$P$, as we
show in Sections~\ref{s41} and \ref{s53}.  Thus, they are badly
behaved and do not satisfy the transfer principle of Robinson's
framework (see \cite{BK20a}, Section\;3.2) and are therefore
\emph{parasitic} in the sense of Clendinnen \cite{Cl89}.%
\footnote{See note \ref{f5}.}
%

%
%
The significance of transfer and related principles both in the
current practice of non-Archimedean mathematics based upon Robinson's
framework (see \cite{BK20a}, Sections 3.1 and 3.2) and in the
historical development of mathematical theories with infinitesimals
(see \cite{BK20a}, Section\;3.3) is sufficient reason to recast AP's
theorems as a warning against the use of external probabilities in
hyperreal modeling.  Thus a careful analysis of Prussian theorems
enables a meaningful criterion for the rejection of AP's
\emph{underdetermination} charge.

For the purposes of the discussion that follows, recall that the
transfer principle entails that any internal probability measure on a
hyperfinite sample space~$\Omega$ is \emph{hyperfinitely additive},
i.e., that for every internal $A\subseteq\Omega$, one has~$P(A) =
\sum_{\omega \in A} P(\{\omega\})$.  Consequently, if a probability
measure is not hyperfinitely additive, then it is not internal.

\subsection{The first theorem of Pruss}
\label{s41}

We begin our analysis with AP's first theorem.  AP is not clear on the
domain of~$P$; here we will assume that~$P$ is an internal probability
measure defined over a hyperfinite set~$\Omega$ and that
$P(\{\omega\}) \approx 0$ for all~$\omega \in \Omega$.  This
hypothesis is not restrictive, since these measures are general enough
to represent every non-atomic\footnote{A real-valued probability
function is called \emph{non-atomic} if and only if it assigns
measure~$0$ to every singleton. If a real-valued probability
measure~$P$ can be decomposed into a sum of a non-atomic
measure~$P_{\rm{na}}$ and a discrete measure~$P_{\rm{d}}$, then our
Theorem \ref{theorem 1} can still be applied to the hyperfinite
representatives of~$P_{\rm{na}}$.}  real-valued probability measure
\cite[Theorem\;2.2, p.\;6]{bbd}.  We define~$P_\alpha$ as in
\eqref{e41}.  Then the following theorem holds.

\begin{theorem}
\label{theorem 1}
Let~$\Omega$ be a hyperfinite set and let~$P\colon
\ns{\mathcal{P}}(\Omega) \rightarrow \astr$ be an internal probability
measure that satisfies~$P(\{\omega\}) \approx 0$ for all~$\omega \in
\Omega$.  If~$\alpha \ne 1$ then~$P_\alpha$ is external.
\end{theorem}

The significance of such externality can be appreciated in light of
the following fact.  If $\{X_i\colon i<H\}$ is an internal sequence of
sets of infinite hyperfinite length $H$, and $P$ an internal measure,
then the sum $\sum_i P(X_i)$ is well defined in $\astr$; but if $P$ is
not internal then generally speaking $\sum_i P(X_i)$ cannot be
reasonably defined at all.

\begin{proof}[Proof of Theorem~\ref{theorem 1}]
Since~$P$ is internal, it is hyperfinitely additive.  As a
consequence,
\[
\sum_{\omega \in \Omega} P_{\alpha}(\{\omega\}) = \sum_{\omega \in
\Omega} \alpha P(\{\omega\}) = \alpha \sum_{\omega \in \Omega}
P(\{\omega\}) = \alpha.
\]
Since~$P_\alpha(\Omega) = 1$, the probability~$P_\alpha$ is
hyperfinitely additive if and only if~$\alpha = 1$, so that if~$\alpha
\ne 1$ then~$P_\alpha$ cannot be hyperfinitely additive and, as a
consequence, it is external.
\end{proof}

\subsection{The~$\Omega$-limit axiom}
\label{s42}

A similar argument refutes AP's interpretation of his second theorem;
see Section~\ref{s53}.  Before showing that the probability
measures~$P_\phi$ obtained by AP are external, it will be
convenient to recall the $\Omega$-limit axiom of Benci et
al.\;\cite{benci2013} and some of its consequences.  In this
subsection,~$\Omega$ will be a set of classical mathematics.  Define
also~$\Lambda=\{A\subseteq\Omega\colon|A|\in\N\}$. Thus if~$\lambda
\in \Lambda$, then~$\lambda$ is a finite set.

The~$\Omega$-limit is a notion of limit governed by the following definition.

\begin{definition}
Let~$\Omega$ be an infinite set and~$F$ an ordered field~$F \supset
\R$.  An \emph{$\Omega$-limit} in~$F$ is a correspondence that
associates to every function~$f\colon \Lambda \rightarrow \R$, an
element of~$F$, denoted by~$\olim f$, in such a way that the following
properties hold:
\begin{enumerate}
\item if there is a~$\overline{\lambda} \in \Lambda$ with~$f(\lambda)
= c \in \R$ for every~$\lambda \supseteq \overline{\lambda}$,
then~$\olim f(\lambda) = c$;
\item for every~$f, g \colon \Lambda \rightarrow \R$, one has
\begin{itemize}
\item~$\olim(f(\lambda)+g(\lambda)) = \olim f(\lambda) + \olim
g(\lambda)$, and
\item~$\olim(f(\lambda)\cdot g(\lambda)) = \olim f(\lambda) \cdot
\olim g(\lambda).$
\end{itemize}
\end{enumerate}
\end{definition}
It is possible to obtain an~$\Omega$-limit by a suitable ultrapower
construction.  If one defines
%
%
$F=\R^\Lambda/\mathcal{U}$, where~$\mathcal{U}$ is a fine and free
ultrafilter over~$\Lambda$,%
\footnote{An ultrafilter~$\mathcal{U}$ over~$\Lambda$ is \emph{fine}
whenever for every~$\lambda \in
\Lambda=\{A\subseteq\Omega\colon|A|\in\N\}$, one has~$\{ L \subseteq
\Lambda \colon \lambda \in L \} \in \mathcal{U}$.  Such ultrafilters
were referred to as \emph{adequate} in earlier literature; see e.g.,
Kanovei--Reeken \cite[p.\;143]{KR}.  A \emph{free} ultrafilter is an
ultrafilter that does not have a~$\subseteq$-least element.}
then~$F$ is a field of hyperreal numbers.  An~$\Omega$-limit over~$F$ can then
be obtained by setting~$\olim f(\lambda) = [f]_{\mathcal{U}}$.
	
A probability function~$P \colon \mathcal{P}(\Omega) \rightarrow \astr$
satisfies the~$\Omega$-limit axiom if and only if there exists an
$\Omega$-limit such that~$P(A) = \olim P(A|\lambda)$ for every~$A
\subseteq \Omega$.
	
Using the~$\Omega$-limit axiom it is possible to define a notion of
infinite sum for~$P$.  Thus, for every~$A \subseteq \Omega$, the sum
$\sum_{\omega \in A} P(\omega)$ is defined as~$\olim
\left(\sum_{\omega \in A \cap \lambda} P(\omega) \right)$; see also
\cite[p.\;6]{benci2018}.

Since a probability that satisfies the~$\Omega$-limit axiom is defined
over a classical set~$\Omega$, the non-Archimedean probabilities that
satisfy the~$\Omega$-limit axiom are external
\cite[p.\;25]{benci2018}. However, this is only due to the
choice by Benci et al.\ 
%
%
of working with a sample space that is not internal.  In fact,
we will now show that non-Archimedean probabilities that satisfy the
$\Omega$-limit axiom can be obtained as the restriction of suitable
internal hyperfinite probabilities.  Let~$\Omega \subseteq \R$.%
\footnote{This hypothesis can be avoided by extending the notion of
$\Omega$-limit as shown by Bottazzi \cite[Sections\;4.2 and 4.3]{tesi}.}
We define the set
\begin{equation}
\label{e42}
\Omega_\Lambda = \left\{ \olim f(\lambda) \colon \Lambda
\stackrel{f}{\rightarrow} \R \text{ and } f(\lambda) \in \lambda
\right\}.
\end{equation}
Then~$\Omega_\Lambda$ is an internal hyperfinite set that represents
$\Omega$.  Moreover, it is possible to define an internal, uniform
probability measure
\begin{equation}
\label{e43}
\overline{P} \colon \ns{\mathcal{P}}(\Omega_\Lambda) \rightarrow \astr
\end{equation}
by setting~$\overline{P}(A) = \frac{|A|}{|\Omega_\Lambda|}$ for every
internal~$A \subseteq \Omega_\Lambda$.  With this definition, one has
\[
P(A)=\overline{P}(\asta\cap\Omega_\Lambda)=
\frac{|\asta\,\cap\,\Omega_\Lambda|}{|\Omega_\Lambda|}
\]
for every~$A \subseteq \Omega$.  A proof of these statements can be
obtained from the proof of Theorem 2.2 in \cite{bbd} as well as
\cite[Section 3.6]{benci2018}.  Observe also that, if~$\phi$ is an
automorphism of~$\astr$, then $P_\phi(A) =
\phi\left(\overline{P}(\asta\cap\Omega_\Lambda)\right)$.

\subsection{The second theorem of Pruss}
\label{s53}

As we showed in Section~\ref{s42}, non-Archimedean probabilities that
satisfy the~$\Omega$-limit are restrictions of hyperfinite internal
measures. We can now state our result concerning AP's
probabilities~$P_\phi$.  First, recall that in \cite[Section\;3.6,
Proposition\;1]{Pr18} AP proves that there are uncountably many
nontrivial automorphisms of~$\astr$ that fix~$\R$.  However, these
automorphisms are all external.

\begin{proposition}
\label{proposition 1}
	If~$\phi \colon \astr \rightarrow \astr$ is a field
	automorphism, then either~$\phi$ is the identity or~$\phi$ is
	external.
\end{proposition}
\begin{proof}
	It is well known that the only field automorphism of~$\R$ is the identity.
	Thus, by transfer, the only internal field automorphism of~$\astr$ is the identity.
\end{proof}

Our Proposition \ref{proposition 1} already suggests that probability
measures~$P_\phi$ are external whenever they do not coincide with~$P$.
However, we will prove such a result explicitly, establishing our
counterpart to AP's second theorem.

\begin{theorem}
\label{theorem 2}
%
Let~$P\colon \mathcal{P}(\Omega) \rightarrow \astr$ be a probability
measure that satisfies the~$\Omega$-limit axiom and let
$\Omega_\Lambda$ and~$\overline{P}$ be defined as in \eqref{e42} and
\eqref{e43}.  Suppose also that~$\phi\colon \astr \rightarrow \astr$ is a
field automorphism.  Then the measure~$\overline{P}_\phi = \phi \circ
\overline{P}$ is internal if and only if the restriction of~$\phi$ to
the range of\,~$\overline{P}$ is the identity.%
%
%
\end{theorem}
\begin{proof}
Recall that the range of the measure~$\overline{P}$ is the hyperfinite
set~$S=\left\{0, \frac{1}{n}, \frac{2}{n},\ldots,\frac{k}{n},
\ldots,1\right\}$, where~$n = |\Omega_\Lambda|$.  Let~$\psi$ be the
restriction of~$\phi$ to the range of~$\overline{P}$.  If~$\psi$ is
the identity, then~$\overline{P}_\phi = \overline{P}$.

Suppose~$\psi$ is an internal automorphism different from the
	identity.  Then
%
%
	the set~$\left\{0, \frac{1}{\psi(n)},
	\frac{2}{\psi(n)},\ldots,\frac{\psi(k)}{\psi(n)}, \ldots, 1
	\right\}$ is internal.  Multiplying every element of this set
	by~$\psi(n)$, we would obtain that the set~$\left\{0, 1,
	2,\ldots,\psi(k), \ldots, \psi(n) \right\}$ is internal, as
	well.  Let~$k$ be the least number such that~$\psi(k)\ne
	k$. Then~$\psi(k-1) = k-1$, so~$\psi(k) = \psi(k-1+1) =
	\psi(k-1)+\psi(1)=k-1+1=k$, a contradiction.  

Thus
%
%
$\psi$ is necessarily external (or the identity).  It remains to show
that~$\psi \circ \overline{P}$ is external.  Since~$\overline{P}$ is
an internal mapping of the hyperfinite set
$^\ast\mathcal{P}(\Omega_\Lambda)$ onto~$S=\left\{0, \frac{1}{n},
\frac{2}{n},\ldots,1\right\}$, it has an internal right inverse
$\chi\colon S\rightarrow {}^\ast\mathcal{P}(\Omega_\Lambda)$ such that
$\overline{P} \circ \chi~$ is the identity on~$S$.  If~$\psi \circ
\overline{P}$ were internal, then~$\psi \circ \overline{P} \circ \chi
= \psi$ would be internal, contrary to our hypothesis.
%
%
\end{proof}


\subsection
{Internal probability measures are not underdetermined}
\label{notunder}

Our Theorems \ref{theorem 1} and \ref{theorem 2} indicate that the
\emph{underdetermination} issue raised by AP is present only
whenever one considers external probability measures in addition to
the internal ones.  Significantly, such external probability measures
do not exist in Nelson's Internal Set Theory (see \cite{BK20a},
Section\;3.2).  Notice also that working with internal probability
measures is not restrictive, since, as we have shown, non-Archimedean
functions that satisfy the~$\Omega$-limit can be obtained as the
restriction of suitable internal hyperfinite probabilities.

It should be noted that we are not claiming that external measures are
not useful for hyperreal models.  In fact, the external
non-Archimedean probabilities by Benci et al.\ and the \emph{Loeb
measures}, that are obtained from suitable internal measures, play a
relevant role for the development of infinitesimal probabilities and
for the hyperreal measure theory, respectively.%
\footnote{Recall that the Loeb measure of an internal measure~$\mu$ is
obtained by composing~$\mu$ with the (external) standard part
function.  By the Caratheodory extension theorem, the resulting
pre-measure is then extended to a measure defined on an external
$\sigma$-algebra, that is usually called the \emph{Loeb measure}
associated to~$\mu$.}
What we showed is that the external probabilities proposed by AP do
not have even basic properties such as hyperfinite additivity. As a
consequence, these external probabilities are clearly inferior to
their internal counterparts.

In conclusion, once a hyperfinite sample space~$\Omega$ is determined,
there is only one internal, uniform and regular probability measure
$P$ over~$\Omega$.  As a consequence, internal uniform probability
measures with hyperfinite support are not \emph{underdetermined},
contrary to what was alleged by Pruss.%
\footnote{This argument and the possibility of uniquely specifying a
hyperreal field, as discussed in Section 2.1 of \cite{BK20a}, refute
Pruss' claim that hyperreal probabilities are underdetermined both in
the choice of a specific hyperreal field and in the choice of the
infinitesimal probability of singletons.}

Indeed, Pruss' pair of theorems can be interpreted as a warning
against the use of external measures in hyperreal probabilities.  As
already noted in Section~\ref{s45}, Pruss envisions the possibility of
introducing additional constraints.  The discussion above suggests
that the very first constraint one should envision is the
\emph{natural} (and even \emph{obvious}, to a practitioner of
mathematics in Robinson's framework) constraint of being
\emph{internal}.

In an analysis of the phenomenon of underdetermination, Clendinnen
speaks of ``empirically equivalent alternatives [that are] parasitic
on [the original] theory'' \cite[p.\;63]{Cl89}, an apt description of
Pruss' external measures of the form \eqref{e41} and \eqref{e411}.

\section{Probability measures on transferless fields} 
\label{s6}

Pruss claims that some extensions of the real numbers with
infinitesimals, such as the surreal numbers, the fields of Laurent
series or the Levi-Civita field, might provide a non-arbitrary choice
of infinitesimals for the development of non-Archimedean
probabilities.
%


%

%

Due to the absence of a transfer principle, already discussed in
\cite[Section\;3.4]{BK20a}, there are some obstacles to developing a
probability theory over these non-Archimedean extensions of the real
numbers.  We will discuss two examples in detail in Sections~\ref{sec
surreals} and\;\ref{s62}.

\subsection
{Probability measures on the surreal numbers}
\label{sec surreals}

Pruss suggests that Conway's surreal numbers (usually denoted
\textbf{No}) might be suitable for the development of infinitesimal
probabilities, since
\begin{quote}
	Whatever you can do with hyperreals, you
	can do with surreals, since any field of hyperreals can be embedded in the surreals. \cite[Section\;2]{Pr18}
\end{quote}
Pruss seeks support for such a claim in Ehrlich \cite{ehrlich}.  In
fact, the result \cite[Theorem\;20]{ehrlich} entails that in the von
Neumann--Bernays--G\"odel set theory with global choice,
%
%
there exists a unique (up to an isomorphism) structure~$(\R, \R^\ast,
\ast)$ such that~$\R^\ast$ satisfies Keisler's axioms for hyperreal
number systems and~$\R^\ast$ is a proper class 
isomorphic to the class of surreal numbers.

Ehrlich observes, furthermore, that every field of hyperreal numbers
is isomorphic to an initial segment of the surreal numbers:
\begin{quote}
Since every real-closed ordered field is isomorphic to an initial
subfield of \textbf{No}, the underlying ordered field of any hyperreal
number system is likewise isomorphic to an initial subfield of
\textbf{No}.%
\footnote{It should be noted that the property of being an ordered
field is a small fraction of the properties of the reals required to
develop any substantial calculus.  The surreals lack many relevant
properties even with regard to their collection of natural numbers;
see \cite{BK20a}, note~28.}
For example, the familiar ultrapower construction of a
hyperreal number system as a quotient ring of~$\R^\N$ (modulo a given
nonprincipal ultrafilter on~$\N$) is isomorphic
to~$\mathbf{No}(\omega_1)$ {\ldots} assuming [the continuum
hypothesis] {\ldots} Similarly, if we assume there is an uncountable
inaccessible cardinal, ~$\omega_{\alpha}$ being the least, then
~$\mathbf{No}(\omega_\alpha)$ {\ldots} is isomorphic to the underlying
ordered field in the hyperreal number system employed by Keisler in
his \emph{Foundations of Infinitesimal Analysis}.  \cite[Section\;9,
p.\;35]{ehrlich}
\end{quote}

Indeed, it is only the isomorphisms mentioned by Ehrlich that endow
the surreal numbers (or their initial segments) with the transfer
principle of Robinson's framework.%
\footnote{As well as with the natural numbers and all the associated
structure over $\langle\R;\N\rangle$ necessary to develop calculus and
measure theory.}
In addition, the~$\ast$-map allows for the extension of functions in a
way that the \emph{simplicity hierarchical structure} of the surreal
numbers is not yet able to handle~\cite{workshop}.

If one refrains from exploiting the transfer principle of Robinson's
framework, then the current developments of analysis over the surreal
field are fairly limited.  For instance, it is still an open problem
to define an integral over \textbf{No}, despite some limited results
obtained in the last two decades (Costin et al.\ \cite{integral},
2015; Fornasiero \cite{fornasiero}, 2004).  As a consequence, it is
not currently possible to develop a measure theory, let alone a
probability theory, over the surreal numbers.

\subsection
{Probability measures on the Levi-Civita field}
\label{s62}

At the other end of the spectrum, Pruss advocates the use of the
``elegantly small'' Levi-Civita field~$\civita$. This extension of the
real numbers was introduced by Levi-Civita in \cite{civita1, civita2}.
It is the smallest non-Archimedean ordered field extension of the
field~$\R$ of real numbers that is both real closed and complete in
the order topology.%
\footnote{As opposed to other forms of completeness.}
For an introduction, we refer to Lightstone--Robinson \cite{robinson
civita} and to Berz--Shamseddine \cite{shamseddine berz analysis}.

The Levi-Civita field is defined as the set
\[
\civita = \left\{ x\in \R^\Q\colon \forall q \in \Q\ \mathrm{supp} (x)
\cap (-\infty,q] \text{ is finite} \right\}
\]
together with the pointwise sum, and the product defined by the
formula
\begin{equation}
\label{e41b}
	(x \cdot y)(q) = \sum_{q_1+q_2=q} x(q_1) \cdot y(q_2).
\end{equation}
If one defines the element~$d \in \civita$ by posing
$$
	d(q) = \left\{
	\begin{array}{ll}
	1 & \text{if } q = 1\\
	0 & \text{if } q \ne 1,
	\end{array}
	\right.
$$
%
then every number in~$\civita$ can be written as a formal sum
$
	x = \sum_{q \in \Q} a_q d^q.
$
In this sum the set
$
	Q(x) = \{ q \in \Q \colon a_q \ne 0\}
$
has the property that the intersection~$Q(x) \cap (-\infty,q]$ is finite for every~$q \in \Q$.%
\footnote{Pruss incorrectly states that only finitely many of the
$a_q$ are nonzero \cite[Section\;2]{Pr18}.}
Elements of~$\R$ can be identified with those elements in~$\civita$
whose support is a subset of the singleton~$\{0\}$.  The
field~$\civita$ can be linearly ordered%
\footnote{\label{foot37}It is customary to define the order
on~$\civita$ in a way that the number~$d$ is an infinitesimal.}
and is sequentially complete in the order topology, but due to the
presence of infinitesimal elements it is totally disconnected%
\footnote{This property is shared by every non-Archimedean extension
of~$\R$. However, fields of hyperreal numbers of Robinson's framework
overcome this limitation by working with internal sets and
functions. For instance, in both the Levi-Civita field and in~$\astr$
the series~$\sum_{n\in\N} r$ is not defined unless~$r =0$, but a sum
of hyperfinitely many copies of~$r \in \astr$ is
well-defined. Similarly, in the Levi-Civita field one should take into
account that the function~$f$ defined by~$$f(x) =
\left\{\begin{array}{ll} 1 & \mathrm{if}\ x \approx 0 \\ 0 &
\mathrm{if}\ x \not\approx 0\end{array} \right.$$
%
%
is differentiable at every point in~$\civita$, but it does not satisfy
e.g. the Intermediate Value Theorem or the Mean Value Theorem.
Meanwhile, the counterpart of this function in Robinson's framework is
external, so it does not provide a counterexample to the Intermediate
Value Theorem or to the Mean Value Theorem.} 
(the topological properties of~$\civita$ are discussed in detail by
Shamseddine \cite{shamseddine 2010}).

If one considers the language~$\lang$ of ordered fields,%
\footnote{Namely~$\lang = \{ +, -, 0, \cdot, \mbox{}^{-1}, 1, < \}$,
where~$+$ and~$\cdot$ are binary functions,~$-$ and~$\mbox{}^{-1}$ are
unary functions,~$0$ and~$1$ are constant symbols, and~$<$ is a binary
relation.}
then the real numbers~$\R$, fields of hyperreal numbers and the
Levi-Civita field~$\civita$ are~$\lang$-structures which are models of
the model-complete%
\footnote{A theory in a language~$L$ is \emph{model-complete} if for
every pair~$M$ and~$N$ of models of the theory, if there is an
embedding~$i \colon M \hookrightarrow N$, then the embedding is
elementary. An embedding is elementary whenever for every first-order
formula~$\phi$ in the language~$L$,~$M \models \phi(a_1, \ldots, a_n)$
if and only if~$N \models \phi(i(a_1), \ldots, i(a_n)$.}
$\lang$-theory of real closed ordered fields
\cite{keislerchang,modeltheory}.  However, the Levi-Civita field is
not elementarily equivalent to~$\R$, and the only result analogous to
the transfer principle between~$\R$ and~$\civita$ is fairly
limited. In particular, only locally analytic functions can be
extended from real closed intervals to closed intervals in~$\civita$
in a way that preserves elementary properties~\cite[Sections 3,
4]{Bo18}.

Consider now the uniform measure~$\m$ on the Levi-Civita field studied
by Berz--Shamseddine (\cite{berz+shamseddine2003}, 2003), Shamseddine
(\cite{shamseddine2012}, 2012), and Bottazzi (\cite{bottazzi
forthcoming}, 2020).  A set is measurable with respect to~$\m$ if it
can be approximated arbitrarily well by intervals, in analogy with the
Lebesgue measure.  However, due to the properties of the topology of
the Levi-Civita field, the measure~$\m$ turns out to be rather
different from the Lebesgue measure over~$\R$. For instance,
in~$\civita$ the complement of a measurable set is not necessarily
measurable.  As a consequence, the family of measurable sets is not
even an algebra.

As with the Lebesgue measure, the family of measurable functions is
obtained from a family of simple functions.  For the Lebesgue measure,
these are the step functions; however a similar choice in the
Levi-Civita field would lead to a very narrow class of measurable
functions. A more fruitful choice for the Levi-Civita field is to
define the simple functions as the analytic functions over a closed
interval.

Regardless of the choice of simple functions, Shamseddine and Berz
proved that any measurable function on a set~$A\subseteq \civita$ is
locally simple almost everywhere on~$A$ \cite[Proposition 3.4,
p.\;379]{berz+shamseddine2003}. As a consequence, an absolutely
continuous probability measure over~$\civita$ can only have a locally
analytic density function; similarly, the only real probability
functions that can be extended to probability functions over the
Levi-Civita field are locally analytic.  A partial converse of this
negative result shows that measurable functions are not expressive
enough to approximate with an infinitesimal error a real probability
that is not locally analytic at any point of its domain
\cite[Proposition 3.16]{bottazzi forthcoming}.

We now turn our attention to discrete probability measures defined
over the Levi-Civita field.  We note that this field has no notion of
hyperfiniteness or of sum over sets that are not countable.  Moreover,
if~$h \in \civita$ is an infinitesimal, the limit~$\lim_{n \rightarrow
\infty} \sum_{i =0}^n h = \lim_{n \rightarrow \infty} nh$ is not
defined, since the sequence~$\{nh \}_{n \in\N}$ does not converge
in~$\civita$.  This is a significant difference with hyperreal-valued
measures, where one can sum a internal hyperfinite family of
infinitesimal terms and obtain a well-defined hyperreal result.

As a consequence, whenever~$\Omega \subseteq \civita$ is not a finite
set, there cannot be a regular probability measure~$\m$ over~$\Omega$
that assigns the same probability to every~$\omega\in\Omega$.  In this
regard,~$\civita$-valued discrete probability measures do not offer
any improvement upon real-valued discrete probability measures.

These properties of the uniform measure and of the discrete
probability measures over the Levi-Civita field impose serious
limitations on the use of the Levi-Civita field as the target space of
a probability measure, be it absolutely continuous, discrete, or a
combination of the two.

\subsection
{Probability measures in other fields with infinitesimals}

The Hahn field, obtained from functions $x \in \R^{\Q}$ with
well-ordered support with the pointwise sum and the product defined as
in formula~\eqref{e41b}, shares some properties with the Levi-Civita
field.  Nevertheless, there is no measure theory over the Hahn field
yet, even though some preliminary results on the convergence of power
series were obtained by Flynn and Shamseddine in (\cite{Fl19}, 2019).

Kaiser recently developed a uniform measure over a class of
non-Archi\-me\-dean real closed fields (\cite{kaiser}, 2018).
However, this measure is only defined for semialgebraic sets. This
condition is even more restrictive than the one imposed on the
measurable sets in the Levi-Civita field (see Section~\ref{s62}).
Note also that the resulting measure is only finitely additive.

The picture that emerges is very different from the one sketched by
Pruss.  Robinson's hyperreal fields have many advantages over other
non-Archimedean extensions of~$\R$, as already realized by Fraenkel
(see \cite{BK20a}, Section\;3.3), and it is unlikely that the gap
between these theories will be closed any time soon.

\section{Conclusion} 

In his \emph{Synthese} paper \cite{Pr18}, Pruss claims that the
infinitesimal probabilities of Robinson's framework for analysis with
infinitesimals are \emph{underdetermined}.  His claim hinges upon
\begin{enumerate}
\item
\label{pr1}
some models of infinitesimal probabilities, and
\item
\label{pr2}
a pair of theorems that entail the existence of uncountably many
infinitesimal probabilities that yield the same decision-theoretic
comparisons as the original one.
\end{enumerate}
In Section\;\ref{sec spinners}, we addressed issue \eqref{pr1} by
showing that
%
%
proper hyperfinite models avoid the \emph{underdetermination} problem.
In Section\;\ref{sec underdetermination} we focused on Prussian
theorems mentioned in item~\eqref{pr2}, and proved that all of the
additional infinitesimal probabilities obtained by Pruss are
external. As a consequence, once a hyperfinite sample space has been
chosen, there is only one internal probability measure over it. The
results of Sections\;\ref{sec spinners} and \ref{sec
underdetermination} suggest that the \emph{underdetermination}
critique of Pruss is limited to external entities.%
\footnote{We remark that, while the probability functions that satisfy
the~$\Omega$-limit axiom are external, they are obtained as the
restriction of suitable internal hyperfinite functions, as shown in
Section~\ref{s42}.  Thus they are also impervious to Pruss'
\emph{undetermination} claim.}
Thus, working with internal sets and functions in Robinson's framework
dissolves the underdetermination objection.  While recognizing that it
may be possible to narrow down the choice of infinitesimal probability
using additional constraints, Pruss fails to consider the natural
\emph{internal} constraint.  Pruss could respond by arguing that
internality is not the kind of criterion he had in mind; however, the
point remains that internality is such an obvious choice that it
should have been addressed one way or another.  The fact that the
issue is not examined in his paper constitutes a serious shortcoming
of his analysis.

A would-be critic of Robinson's framework could then retreat to an
even more limited objection to the effect that the choice of the
sample space is underdetermined.  However, in the discussion of the
hyperfinite models we showed the following:
\begin{itemize}
	\item this choice is underdetermined in the Archimedean case,
	as well; however, this underdetermination is not problematic,
	since different models (be they Archimedean or hyperfinite)
	are compatible, as discussed in Sections\;\ref{sec archimedean
	spinners 2} and \ref{sec hyperfinite spinners}.
	\item it is also possible to specify additional criteria for
	the choice of a hyperfinite model: for instance, it is
	possible to preserve rotational symmetry on a non-trivial
	family of sets or coherence with uncountably many real-valued
	measures, as shown in Section\;\ref{s45}.  Thus the
	possibility of working with hyperfinite models that improve
	upon the properties of the Archimedean ones should be regarded
	as an advantage of hyperreal probabilities.
\end{itemize}
In his critique, Pruss also suggests that other fields with
infinitesimals are more suitable for the development of infinitesimal
probabilities.  However, in \cite{BK20a}, Sections\;3.3 and\;3.4 we
showed that this claim ignores the Klein--Fraenkel criteria for the
usefulness of a theory with infinitesimals.  In addition, in
Section\;\ref{s6} we showed that probabilities taking values in the
surreal numbers and in the Levi-Civita field are less expressive than
real-valued probabilities, mainly due to the absence of a transfer
principle for these structures.  Moreover, the absence of a
comprehensive transfer principle makes such fields vulnerable to
Theorem 1 of Pruss.  In contrast to hyperreal fields, transferless
fields do not possess a notion of internality and are thus unable to
escape underdetermination.

\section*{Acknowledgments}

We are grateful to Karel Hrbacek and Vladimir Kanovei for insightful
comments on earlier versions that helped improve our article, and to
anonymous referees for constructive criticism.  The influence of
Hilton Kramer (1928--2012) is obvious.

\end{document}